\documentclass[11pt]{article}
\usepackage[margin=1in]{geometry}
\usepackage[utf8]{inputenc}
\usepackage[T1]{fontenc}
\usepackage[english]{babel}
\usepackage{lmodern}
\usepackage{graphicx}
\usepackage{wrapfig}
\usepackage{subfigure}
\usepackage{caption}
\usepackage{paralist}
\usepackage{enumitem}  
\usepackage{amssymb}

\usepackage[linesnumbered,ruled,vlined]{algorithm2e}

\graphicspath{{./figures/}}

\makeatletter
\renewcommand{\p@enumii}{}
\makeatother

\newcommand{\titel}{
	On the spanning structure hierarchy of 3-connected planar graphs
	}
\usepackage[usenames]{color}
\definecolor{hellblau}{rgb}{0.2,0.4,1} 
\definecolor{dunkelblau}{rgb}{0,0,0.8}
\definecolor{dunkelgruen}{rgb}{0,0.5,0}
\usepackage[
	pdftex,
	colorlinks,
	linkcolor=dunkelblau,
	urlcolor=dunkelblau,
	citecolor=dunkelgruen,
	bookmarks=true,
	linktocpage=true,
	pdfauthor={On-Hei Solomon Lo},
	pdfsubject={}
]{hyperref}
\usepackage{pdfpages}
\usepackage{amsmath}
\usepackage{amsthm}
\allowdisplaybreaks
\usepackage{amsfonts}
\theoremstyle{plain}
	\newtheorem{satz}{Satz}[]
	\newtheorem{theorem}[satz]{Theorem}
	\newtheorem{lemma}[satz]{Lemma}
	
	\newtheorem{proposition}[satz]{Proposition}
		
\theoremstyle{remark}
	
\theoremstyle{definition}
	
	\newtheorem{conjecture}[satz]{Conjecture}
	\newtheorem{problem}[satz]{Problem}

\makeatletter
\newcommand{\setword}[2]{%
	\phantomsection
	#1\def\@currentlabel{\unexpanded{#1}}\label{#2}%
}
\makeatother

\begin{document}
	\title{\titel}
		\author{
		On-Hei Solomon Lo\thanks{School of Mathematical Sciences, Xiamen University, Xiamen 361005, PR China. This work was partially supported by NSFC grant 11971406.}
		}
	\date{}
	\maketitle

\begin{abstract}
The prism over a graph $G$ is the Cartesian product of $G$ with the complete graph $K_2$. $G$ is prism-hamiltonian if the prism over $G$ has a Hamilton cycle. A good even cactus is a connected graph in which every block is either an edge or an even cycle, and every vertex is contained in at most two blocks. It is known that good even cacti are prism-hamiltonian. Indeed, showing the existence of a spanning good even cactus has become one of the most common techniques in proving prism-hamiltonicity. \v{S}pacapan asked whether having a spanning good even cactus is equivalent to having a hamiltonian prism for 3-connected planar graphs. In this article we give a negative answer to this question by showing that there are infinitely many 3-connected planar prism-hamiltonian graphs that have no spanning good even cactus. We also prove the existence of an infinite class of 3-connected planar graphs that have a spanning good even cactus but no spanning good even cactus with maximum degree three.
\end{abstract}

\section{Introduction} \label{sec:intro}

In 1884, in an attempt to solve the Four Color Theorem (which was open then), Tait~\cite{Tait1884} gave a proof which required that every 3-connected 3-regular planar graph is hamiltonian. The proof, however, turned out to be false and the first counterexample to the hypothesis, a non-hamiltonian 3-connected 3-regular planar graph, was constructed by Tutte~\cite{Tutte1946} in 1946. It is also known that there are 3-connected 3-regular planar graphs that have no Hamilton path~\cite{Gruenbaum1962, Brown1961}. 

On the positive side, Tutte~\cite{Tutte1956} showed in 1956 that every 4-connected planar graph does have a Hamilton cycle. Tutte's result was strengthened by Thomassen~\cite{Thomassen1983} who proved that every 4-connected planar graph is Hamilton-connected, that is, any two vertices are connected by a Hamilton path.

Hamilton cycles and Hamilton paths can be generalized by the following notion. A \emph{$k$-walk} is a spanning closed walk that visits every vertex at most $k$ times; a \emph{$k$-tree} is a spanning tree with maximum degree at most $k$. Clearly, a graph has a $k$-walk if it has a $k$-tree. It was shown in~\cite{Jackson1990} that every $k$-walk contains a subgraph that is a $(k + 1)$-tree. In particular, 1-walk and 2-tree are the same notion of Hamilton cycle and Hamilton path, respectively. 

The \emph{prism} over a graph $G$ is the Cartesian product of $G$ and $K_2$, denoted by $G \square K_2$. $G$ is \emph{prism-hamiltonian} if and only if $G \square K_2$ is hamiltonian. The following chain of implications on the existence of spanning structures is well-known:

\begin{center}
	Hamilton cycle $\Rightarrow$ Hamilton path $\Rightarrow$ hamiltonian prism $\Rightarrow$ 2-walk $\Rightarrow$ 3-tree
\end{center}

Barnette~\cite{Barnette1966} showed that every 3-connected planar graph has a 3-tree. Confirming a conjecture of Jackson and Wormald~\cite{Jackson1990}, Gao and Richter~\cite{Gao1994} proved that every 3-connected planar graph has a 2-walk (see also~\cite{Gao1995}). It is natural to ask if all 3-connected planar graphs can reach the level of ``hamiltonian prism'' in the suggested hierarchy. This was formulated as a conjecture by Kaiser, Ryj\'{a}\v{c}ek, Kr\'{a}l', Rosenfeld and Voss~\cite{Kaiser2007}, which was also attributed to Rosenfeld and Barnette in~\cite{Spacapan2021}.

\begin{conjecture}[{\cite[Conjecture~1]{Kaiser2007}}] \label{conj:PHP}
	Every 3-connected planar graph has a hamiltonian prism.
\end{conjecture}

Supporting this conjecture, a number of subclasses were shown to be prism-hamiltonian: 3-connected 3-regular (not necessarily planar) graphs~\cite{Paulraja1993}, Halin graphs~\cite{Kaiser2007}, 3-connected bipartite planar graphs~\cite{Biebighauser2008}, near-triangulations~\cite{Biebighauser2008} and 3-connected planar graphs with minimum degree at least four~\cite{Spacapan2021a}. However, a recent breakthrough by \v{S}pacapan~\cite{Spacapan2021} showed that the conjecture is not true in general.

\begin{theorem}[{\cite{Spacapan2021}}]
	There are infinitely many $3$-connected planar non-prism-hamiltonian graphs.
\end{theorem}

Based on \v{S}pacapan's technique, Ikegami, Maezawa and Zamfirescu~\cite{Ikegami2021} provided various classes of counterexamples with special properties.

A \emph{good even cactus} is a connected graph in which every block is either an edge or an even cycle and every vertex is contained in at most two blocks (see Figure~\ref{fig:fp}). It is known that the prism of any good even cactus is hamiltonian. Therefore, one can assert that a graph is prism-hamiltonian if it has a spanning good even cactus. This strategy has been used in proving prism-hamiltonicity for various planar and non-planar graph classes; we refer to~\cite{Paulraja1993, Kaiser2007, Biebighauser2008, Ozeki2009, Bueno2011, Ellingham2020a, Ellingham2020, Spacapan2021a} for examples.\footnote{Note that some graph classes in the given examples were not explicitly shown to have the property of having a spanning good even cactus, but one may justify it by modifying the corresponding original proof.} It is worth noting that in~\cite{Paulraja1993, Kaiser2007, Ozeki2009, Bueno2011, Ellingham2020a} a more restrictive approach was adopted, namely showing the existence of a spanning good even cactus with maximum degree at most three. This proof technique motivates us to refine the spanning structure hierarchy as follows:

\begin{center}
	Hamilton cycle $\Rightarrow$ Hamilton path \\
	$\Rightarrow$ spanning good even cactus with maximum degree at most three \\
	$\Rightarrow$ spanning good even cactus $\Rightarrow$ hamiltonian prism $\Rightarrow$ 2-walk $\Rightarrow$ 3-tree
\end{center}

Obviously the three new implications hold. There are 3-connected planar graphs showing that the implication from ``Hamilton path'' to ``spanning good even cactus with maximum degree at most three'' is sharp; an example is given in Figure~\ref{fig:NT}. \v{S}pacapan~\cite{Spacapan2021} recently asked whether the implication from ``spanning good even cactus'' to ``hamiltonian prism'' can be reversed for 3-connected planar graphs.

\begin{problem}[{\cite[Problem~3.3]{Spacapan2021}}]
	Prove or disprove the following statement. Every 3-connected planar prism-hamiltonian graph has a spanning good even cactus.
\end{problem}

\begin{figure} [t]
	\centering
	\includegraphics[scale = 0.6]{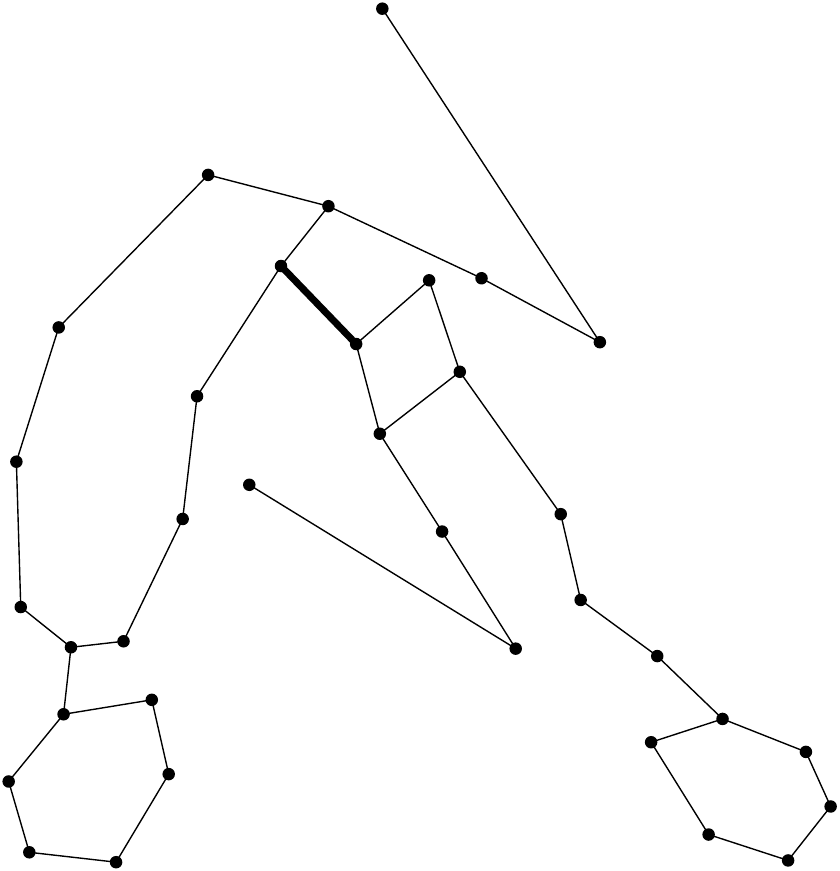}
	\caption{A good even cactus with maximum degree three, which becomes a good even cactus with maximum degree four if the thick edge is contracted.}
	\label{fig:fp}
\end{figure}

\begin{figure} [t]
	\centering
	\includegraphics[scale = 0.8]{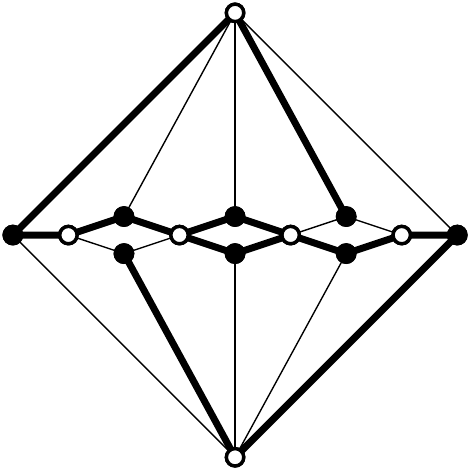}
	\caption{A variation of Herschel's graph which has no Hamilton path (as it has eight components when the six white vertices are removed) and has a spanning good even cactus with maximum degree at most three (thick edges).}
	\label{fig:NT}
\end{figure}

The main purpose of this article is to show that every implication in the new hierarchy proposed above is sharp. Inspired by \v{S}pacapan's counterexamples to Conjecture~\ref{conj:PHP}, we show that there are infinitely many 3-connected planar graphs that have a spanning good even cactus but no such spanning subgraph with maximum degree at most three (Theorem~\ref{thm:atmost3}) and there are infinitely many 3-connected planar graphs that have a hamiltonian prism but no spanning good even cactus (Theorem~\ref{thm:no}), thereby answering the question raised by \v{S}pacapan.

We remark that ``hamiltonian prism'' can be replaced by ``spanning good cactus'' in the hierarchy we consider above. As mentioned in~\cite{Spacapan2021}, the partitioning result given in~\cite{Gao1994} assures that every 3-connected planar graph has a good cactus as a spanning subgraph.

The proofs of our results will be given in the next section. We conclude this section with some terminology and notation.

Let $H$ be a graph and $V$ be a vertex set (not necessarily a subset of $V(H)$). The subgraph of $H$ induced by $V \cap V(H)$ is denoted by $H[V]$. For graphs $H_1$ and $H_2$, $H_1[H_2]$ means $H_1[V(H_2)]$. For any set $U$ of vertices and edges, we use $H - U$ to denote that graph obtained from $H$ by deleting the elements in $U$; we may also write $H - u$ instead of $H - \{u\}$ when $U = \{u\}$. The union $H_1 \cup H_2$ of graphs $H_1$ and $H_2$ is defined to be $(V(H_1) \cup V(H_2), E(H_1) \cup E(H_2))$. Let $H'$ be a subgraph of $H$ and $E \subseteq E(H)$ be an edge set. We may denote by $H' \cup E$ the union of $H'$ and the subgraph of $H$ induced by $E$. Let $u, v$ be two vertices in a connected graph $H$. The graph $H[u, v]$ is defined to be the minimal union of blocks of $H$ such that $H[u, v]$ is connected and contains vertices $u$ and $v$.
For any graph $H$ and any $v \in V(H)$, let $H^i$ be a copy of $H$, we may denote by $v^i$ the duplicate of $v$ in $H^i$.

A \emph{cactus} $Q$ is a connected graph such that every block of $Q$ is either an edge or a cycle. For any $v \in V(Q)$, the \emph{block degree} $b_Q(v)$ of $v$ in $Q$ is defined to be the number of blocks of $Q$ that contain $v$. We call a block of $Q$ that is an edge (a cycle) an \emph{edge block} (a \emph{cycle block}).  We say that $Q$ is \emph{even} if every cycle block of it is an even cycle. A path $P$ in $Q$ is an \emph{edge path} if every edge of $P$ is an edge block of $Q$. A cactus $Q$ is \emph{good} if $b_Q(v) \le 2$ for any $v \in V(Q)$. 
Note that if we delete some vertex from a good even cactus, the new components are even cacti but need not be good anymore. For this reason we introduce two more types of cacti as follows:
\begin{itemize}
	\item Let $P$ be an edge path in a cactus $Q$. We say that $Q$ is a \emph{$P$-good} cactus if~(i) $b_Q(v) \le 2$ for any vertex $v$ that is not an internal vertex of $P$ and~(ii) $b_Q(v) \le 3$ for any internal vertex of $P$. 
	\item Let $P_1$ and $P_2$ be two edge paths in a cactus $Q$ that have at most one common vertex. Then $Q$ is a \emph{$\{P_1, P_2\}$-good} cactus if~(i) $b_Q(v) \le 2$ for any vertex $v$ that is neither an internal vertex of $P_1$ nor an internal vertex of $P_2$,~(ii) $b_Q(v) \le 3$ for any vertex $v$ that is an internal vertex of either $P_1$ or $P_2$ (but not for both) and~(iii) $b_Q(v) = 4$ for any vertex $v$ that is an internal vertex for both paths $P_1$ and $P_2$. 
\end{itemize}
A cactus $Q$ is \emph{$1$-good} if there exists an edge path $P$ in $Q$ such that $Q$ is $P$-good; \emph{$2$-good} if there exist two edge paths $P_1$ and $ P_2$ in $Q$ sharing at most one vertex such that $Q$ is $\{P_1, P_2\}$-good. 

We always assume that the complete graph $K_2$ is on $\{\alpha, \beta\}$. The prism $H \square K_2$ over $H$ is defined to be the graph on $V(H \square K_2) := V(H) \times \{\alpha, \beta\}$ such that $(u, \gamma) (v, \delta)$ are adjacent if and only if~(i) $uv \in E(H)$ and $\gamma = \delta$ or~(ii) $u = v$ and $\gamma \neq \delta$.
For any $\gamma \in V(K_2)$, we denote by $\bar{\gamma}$ the vertex of $K_2$ other than $\gamma$. 
Let $S$ be a subgraph of $H \square K_2$. 
The reflection $R$ of $S$ is a graph defined as follows: (i) $(u, \gamma) \in V(R)$ if and only if $(u, \bar{\gamma}) \in V(S)$; (ii) $(u, \gamma) (v, \delta) \in E(R)$ if and only if $(u, \bar{\gamma}) (v, \bar{\delta}) \in E(S)$.

\section{Results} \label{sec:results}

We first introduce three fragments that we need in our construction. The plane graph $A$ is as depicted in Figure~\ref{fig:A}. Let $n$ be any positive integer. We define $C_n$ to be the cycle of length $2n + 3$ (see Figure~\ref{fig:CD}(a)) and $D_n$ to be the cactus that has exactly two cycle blocks each of them has length $2n + 3$ (see Figure~\ref{fig:CD}(b)). 
We refer to Figures~\ref{fig:A} and~\ref{fig:CD} for names of vertices in $A$, $C_n$ and $D_n$ that are not explicitly defined in the text.

\begin{figure} [t]
	\centering
	\includegraphics[scale = 1.3]{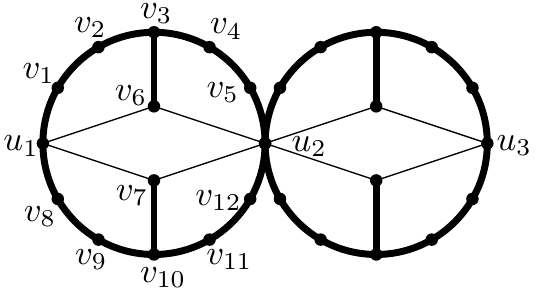}
	\caption{The plane graph $A$ with endvertices $u_1$ and $u_3$. The thick edges induce a spanning good even cactus $K_A$ of $A$ with $b_{K_A}(u_1) = b_{K_A}(u_3) = 1$.}
	\label{fig:A}
\end{figure}

\begin{figure} [t]
	\centering
	\subfigure[]{\includegraphics[scale = 1.3]{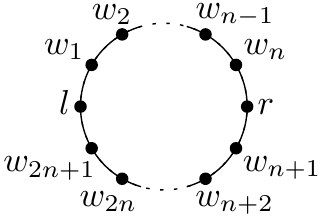}} \label{subfig:C}
	\hfil
	\subfigure[]{\includegraphics[scale = 1.3]{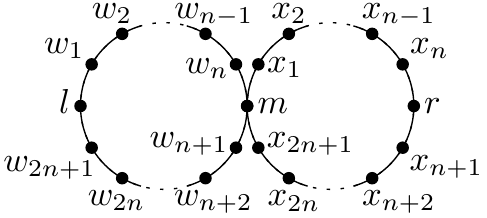}} \label{subfig:D}\\
	\caption{(a) The plane graph $C_n$ with endvertices $l$ and $r$. (b) The plane graph $D_n$ with endvertices $l$ and $r$.} 
	\label{fig:CD}
\end{figure}

We now define our main construction. Let $B$ be either $C_n$ or $D_n$ (which will be fixed throughout the construction). We take eight copies $A^1, \dots, A^8$ of $A$ and seven copies $B^1, \dots, B^7$ of $B$, and form a connected graph $G^-(B)$ from these fifthteen fragments by identifying $u_3^i$ with $l^i$ and identifying $r^i$ with $u_1^{i + 1}$ for every $i \in \{1, \dots, 7\}$. We see $G^-(B)$ as a plane graph by inheriting the plane embeddings of $A$ and $B$ given by Figures~\ref{fig:A} and~\ref{fig:CD}. So the boundary walk around the unbounded face is the union of two edge-disjoint paths with endvertices $u_1^1$ and $u_3^8$, so that the vertices $w_1^1, \dots, w_1^7$ and, if $B$ is $D_n$, the vertices $x_1^1, \dots, x_1^7$ will be contained in the ``upper path'' but not the ``lower path''. The graph $G(B)$ is obtained from $G^-(B)$ and two new vertices $s$ and $t$ by joining $s$ to every vertex in the upper path and $t$ to every vertex in the lower path. We will simply write $G$ and $G^-$ instead of $G(B)$ and $G^-(B)$ if it is clear from the context what $B$ denotes or it causes no ambiguity. It is clear that $G$ is planar. Moreover, it can be shown in exactly the same way as in the proof of~\cite[Lemma~2.5]{Spacapan2021} that $G$ is 3-connected. We conclude with the following lemma.

\begin{lemma}
	Let $n$ be any positive integer and $B$ be either $C_n$ or $D_n$. The graph $G(B)$ constructed above is a $3$-connected planar graph.
\end{lemma}

The main goal of this article is to prove the following two results.

\begin{theorem} \label{thm:atmost3}
	For any positive integer $n$, the $3$-connected planar graph $G(C_n)$ has a spanning good even cactus but no spanning good even cactus with maximum degree three.
\end{theorem}

\begin{theorem} \label{thm:no}
	For any positive integer $n$, the $3$-connected planar graph $G(D_n)$ is prism-hamiltonian but has no spanning good even cactus.
\end{theorem}

The proofs of Theorems~\ref{thm:atmost3} and~\ref{thm:no} will be given in Sections~\ref{sec:atmost3} and~\ref{sec:no}, respectively. Before that, we provide in Section~\ref{sec:pre} a number of lemmas for proving the results.

\subsection{Preliminaries} \label{sec:pre}

We first discuss some properties of the fragments we use in the construction regarding whether they can contain spanning cacti with specified block degree condition.

\begin{lemma} \label{lem:singleeye}
	Let $I := A[\{u_1, u_2\} \cup \{v_1, \dots, v_{12}\}]$. Let $Q$ be a spanning even cactus of $I$ that contains an edge path $P$ with endvertices $u_1$ and $u_2$. Suppose that every vertex of $Q$ that is not an internal vertex of $P$ has block degree at most two in $Q$. We have that $u_1$ and $u_2$ have block degree two in $Q$.
\end{lemma}
\begin{proof}
	By symmetry, we may assume that $P$ is contained in $I [\{u_1, u_2\} \cup \{v_7,  \dots, v_{12}\}]$. Suppose that $b_Q(u_1) = 1$ (reductio ad absurdum). It follows immediately that $u_1 v_1 \notin E(Q)$, $u_1 v_6 \notin E(Q)$, and $v_1 v_2$ and $v_2 v_3$ are two edge blocks of $Q$ (as $Q$ is a spanning cactus). If $v_4$ is contained in any cycle block of $Q \subset I - \{u_1 v_1, u_1 v_6\}$, then that block must be $v_3 v_4 v_5 u_2 v_6 v_3$, which is however impossible since $Q$ is an even cactus. By the same argument, we have that $v_5$ and $v_6$ are not in any cycle block in $Q$.
	
	Note that $Q$ has at least one edge of $v_3 v_6$ and $v_6 u_2$ and at least one edge of $v_3 v_4$ and $v_5 u_2$, since $Q$ is a connected spanning subgraph of $I$. We consider the following two cases. If both $v_3 v_6$ and $v_6 u_2$ are in $E(Q)$, then (depending on $v_3 v_4$ or $v_5 u_2$ is an edge of $Q$) $v_3$ or $u_2$ has to have block degree at least three in $Q$, contradicting the given condition. If exactly one of $v_3 v_6, v_6 u_2$ is in $E(Q)$, then $Q$ must contain $v_3 v_4 v_5 u_2$ as an edge path and $v_3$ or $u_2$ will have block degree at least three in $Q$, again, a contradiction. We thus conclude that $b_Q(u_1) = b_Q(u_2) = 2$.
\end{proof}

\begin{lemma} \label{lem:II}
	Let $Q$ be an even cactus of $A$ that contains an edge path $P$ with endvertices $u_1$ and $u_3$. If $Q$ is $P$-good, then $Q$ is no spanning subgraph of $A$.
\end{lemma}
\begin{proof}
	Suppose that $Q$ is a spanning subgraph of $A$ (reductio ad absurdum). Let $I_1 := A[\{u_1, u_2\} \cup \{v_1, \dots, v_{12}\}]$ and $I_2 := A - \{u_1\} \cup \{v_1, \dots, v_{12}\}$.
	Then, for $i \in \{1, 2\}$, $Q_i := Q[I_i]$ is a spanning even cactus of $I_i$ having $P_i := P[I_i]$ as an edge path with endvertices $u_i$ and $u_{i + 1}$. It is also clear that every vertex of $Q_i$ that is not an internal vertex of $P_i$ is contained in at most two blocks in $Q_i$. Therefore we can apply Lemma~\ref{lem:singleeye} twice to conclude that $b_{Q_1}(u_2) = b_{Q_2}(u_2) = 2$ and hence $b_Q(u_2) = 4$, contradicting our assumption that $Q$ is $P$-good.
\end{proof}

\begin{lemma} \label{lem:O}
	Let $Q$ be a good even cactus of $B$ that contains vertices $l$ and $r$, where $B$ is taken to be $C_n$ or $D_n$ for some positive integer $n$. If $l$ and $r$ have block degree one, then $Q$ is no spanning subgraph of $B$.
\end{lemma}
\begin{proof}
	It is readily to see that $Q$ has no cycle block and hence is an edge path with endvertices $l, r$. In particular, $Q$ is no spanning subgraph of $B$.
\end{proof}

As mentioned before, one cannot guarantee that the components are good cacti after removing some vertices from a good cactus. However, we may characterize, in terms of good, 1-good and 2-good cacti, the components of the graph obtained from some spanning good cactus of $G(B)$ by deleting vertices $s$ and $t$.

\begin{lemma} \label{lem:3or4max3}
	Let $K$ be a good cactus with maximum degree at most three and $s, t$ be two vertices in $K$. One of the following statements holds:
	\begin{enumerate} [label={(\Roman*)}]
		\item $K - s - t$ is a vertex-disjoint union of at most four cacti, at most two of which are $1$-good and the rest are good.
		\item $K - s - t$ is a vertex-disjoint union of at most three cacti, one of which is $2$-good and the rest are good.
	\end{enumerate}
\end{lemma}
\begin{proof}
	As $K$ is a good cactus with maximum degree at most three, no two distinct cycle blocks in $K$ can intersect. 
	
	We now consider the following cases.
	If $s$ and $t$ are in some cycle block $S$ in $K$, then $K - s - t$ is comprised of $b_K(s) + b_K(t) - 2$ good cacti and $k$ 1-good cacti, where $k$ is the number of components of $S - s - t$. As every vertex in $K$ has block degree at most two, $K - s - t$ has at most two good cactus components, at most two 1-good cactus components and no other components. So we are in Case~(I). 
	If $s$ and $t$ are in two cycle blocks $S_s$ and $S_t$ in $K$, respectively, such that $S_s - s$ and $S_t - t$ are in the same component of $K - s - t$, then we are in Case~(II) as $K - s - t$ is comprised of one 2-good cactus and $b_K(s) + b_K(t) - 2 \le 2$ good cacti.
	Otherwise, it is not hard to see that $K - s - t$ has up to three components, at most two of which are 1-good cacti and the rest are good cacti, which is included in Case~(I).
\end{proof}

The following lemma can be proved analogously, thus we omit the proof.

\begin{lemma} \label{lem:3or4}
	Let $K$ be a good cactus and $s, t$ be two vertices in $K$. One of the following statements holds:
	\begin{enumerate} [label={(\Roman*)}]
		\item $K - s - t$ is a vertex-disjoint union of at most four cacti, each of which is good or $1$-good.
		\item $K - s - t$ is a vertex-disjoint union of at most three cacti, one of which is $2$-good and each of the rest is good or $1$-good.
	\end{enumerate}
\end{lemma}

Suppose $G$ has a spanning subgraph $K$ that is a good even cactus. Let $Q$ be a component of $K - s - t$. Let $H$ be any copy of $A$ or $B$ in $G^-$. We say $H$ is a \emph{bag} of $Q$ if $Q$ contains the endvertices of $H$ but not all vertices of $H$, where the endvertices of $A$ (respectively, $B$) are $u_1$ and $u_3$ (respectively, $l$ and $r$). For $i < 1$, we define $l^i$ and $r^i$ to be $u_1^1$; for $i > 7$, we define $l^i$ and $r^i$ to be $u_3^8$. We may choose $0 \le a \le b \le 8$ such that $Q$ is contained in $G^-[l^a, r^b]$ and, subject to this, $b - a$ is minimum. The following three lemmas gives us lower bound on the number of bags of $Q$.

	\begin{lemma} \label{lem:good}
		If $Q$ is a good even cactus, then it has at least $b - a - 1$ bags.
	\end{lemma}
	\begin{proof}
		For every $a < i < b$, $Q [B^i]$ is an even cactus of $B^i$ containing $l^i$ and $r^i$. By the minimality of $b - a$, $l^i$ is contained in some block of $Q [A^i]$ and hence has block degree one in $Q [B^i]$. Similarly, we have $b_{Q[B^i]} (r^i) = 1$. As vertices in $Q [B^i]$ other than $l^i, r^i$ has block degree at most two, we may apply Lemma~\ref{lem:O} to conclude that $Q [B^i]$ does not span $B^i$ and hence $B^i$ is a bag of $Q$. Collecting all these bags for every $a < i < b$, we conclude that $Q$ has $b - a - 1$ bags.
	\end{proof}

	\begin{lemma} \label{lem:1good}
		Suppose $Q$ is a $P$-good even cactus for some edge path $P$ in $Q$. We have that $Q$ has at least $b - a - 2$ bags. Moreover, if $Q$ has no bag and $b = a + 2$, then $P$ is contained in $G^-[r^a, l^b]$.
	\end{lemma}
	\begin{proof}
		Choose $a \le a' \le b' \le b$ such that $P [A^{i + 1}]$ is a path with endvertices $r^i$ and $l^{i + 1}$ for any $a' \le i < b'$, and, subject to this, $b' - a'$ is maximum. 
		
		Note that if $P [A^{i + 1}]$ is a path with endvertices $r^i, l^{i + 1}$, then $Q [A^{i + 1}]$ is a $P [A^{i + 1}]$-good even cactus of $A^{i + 1}$. Therefore, by Lemma~\ref{lem:II}, $Q [A^{i + 1}]$ does not span $A^{i + 1}$ and $A^{i + 1}$ is a bag of $Q$.
		
		By the maximality of $b' - a'$ and the fact that $P$ is a path, we have that $P$ is contained in $G^-[r^{a' - 1}, l^{b' + 1}]$. This and the minimality of $b - a$ imply that for any $i$ with $a < i < a'$ or $b' < i < b$, $Q [B^i]$ is an even cactus of $B^i$ containing $l^i$ and $r^i$. Moreover, $l^i$ and $r^i$ have block degree one in $Q [B^i]$ while all other vertices in $Q [B^i]$ have block degree at most two. By Lemma~\ref{lem:O}, we assure that $B^i$ is a bag of $Q$. Thus we have that $Q$ has at least $(b' - a') + (a' - a - 1) + (b - b' - 1) = b - a - 2$ bags.
		
		Suppose $Q$ has no bag and $b = a + 2$, we claim that $P$ is contained in $G^-[r^a, l^{a + 2}]$. Otherwise, by symmetry, we may assume that $P$ has one endvertex in $B^a - r^a$. Let $p$ be the other endvertex of $P$. If $p$ is in $G^-[l^{a + 1}, r^{a + 2}]$, then it follows from Lemma~\ref{lem:II} that $A^a$ is a bag of $Q$. If $p$ is not in $G^-[l^{a + 1}, r^{a + 2}]$, then, by Lemma~\ref{lem:O}, $B^{a + 1}$ is a bag of $Q$. In any case it contradicts the assumption that $Q$ has no bag. This thus justifies our claim.
	\end{proof}

	\begin{lemma} \label{lem:2good}
		If $Q$ is a $2$-good even cactus, then it has at least $b - a - 3$ bags.
	\end{lemma}
	\begin{proof}
		The proof is similar to what we have done for the previous lemmas. Let $P_1$ and $P_2$ be two edge path in $Q$ having at most one common vertex such that $Q$ is $\{P_1, P_2\}$-good. We may choose $a \le a' \le b' \le a'' \le b'' \le b$ such that $(P_1 \cup P_2) [A^{i + 1}]$ is a path with endvertices $r^i$ and $l^{i + 1}$ for any $i$ with $a' \le i < b'$ or $a'' \le i < b''$, and, subject to this, $b'' - a'' + b' - a'$ is maximum. Now, if $(P_1 \cup P_2) [A^{i + 1}]$ is a path with endvertices $r^i, l^{i + 1}$, then $Q [A^{i + 1}]$ is an even cactus of $A^{i + 1}$ satisfying the block degree condition required by Lemma~\ref{lem:II}, from which it follows that $Q [A^{i + 1}]$ does not span $A^{i + 1}$ and $A^{i + 1}$ is a bag of $Q$.
		
		As $a', b', a'', b''$ are chosen with $b'' - a'' + b' - a'$ maximized and $P_1 \cup P_2$ is either a vertex-disjoint union of two paths or a tree that has at most one vertex of degree larger than two, we have that $P_1 \cup P_2$ is contained in $G^-[r^{a' - 1}, l^{b' + 1}] \cup G^-[r^{a'' - 1}, l^{b'' + 1}]$. For any $i$ with $a < i < a'$ or $b' < i < a''$ or $b'' < i < b$, $Q [B^i]$ is a good even cactus of $B^i$ containing vertices $l^i$ and $r^i$ of block degree one. Applying Lemma~\ref{lem:O}, we have that $B^i$ is a bag of $Q$, and hence $Q$ has at least $(b' - a') + (b'' - a'') + (a' - a - 1) + (a'' - b' - 1) + (b - b'' - 1) = b - a - 3$ bags.
	\end{proof}

\subsection{Proof of Theorem~\ref{thm:atmost3}} \label{sec:atmost3}
	
	In this section we shall show that $G(C_n)$ has a spanning good even cactus, but it does not contain any spanning good even cactus with maximum degree three. Note that $B$ will represent the fragment $C_n$ throughout this section.
	
	A spanning good even cactus $K_A$ of $A$ is depicted in Figure~\ref{fig:A}. We denote by $K_A^i$ the corresponding copy of $K_A$ in $A^i$. It is straightforward to verify that \begin{align*}
	\left( \bigcup_{i = 1}^8 K_A^i \right) &\cup \left( \bigcup_{i = 1, 3} (B^i - l^i w_1^i) \right) \cup \left( \bigcup_{i = 5, 7} (B^i - w_{n + 1}^i r^i) \right) \\
	&\hspace{-3em} \cup \left( \bigcup_{i = 2, 4, 6} (B^i - l^i v_1^i - w_{n + 1}^i r^i) \right)
	\cup \{s l^1, s w_1^1, t l^3, s w_1^3, t w_{n + 1}^5, s r^5, t w_{n + 1}^7, t r^5\}
	\end{align*} is a spanning good even cactus of $G$.
	
	Thus it is left to show that $G$ does not have any spanning good even cactus with maximum degree at most three. Suppose that $G$ has a spanning good even cactus $K$ with maximum degree at most three (reductio ad absurdum). Let $Q_1, \dots, Q_k$ be the components of $K - s - t$ that contain some vertex from $U_2 := \{u_2^1, \dots, u_2^8\}$.
	For every $Q_j$ ($j \in \{1, \dots, k\}$), we choose $0 \le a(j) \le b(j) \le 8$ such that $Q_j$ is contained in $G^-[l^{a(j)}, r^{b(j)}]$ and, subject to this, $b(j) - a(j)$ is minimum. Since the union of $Q_1, \dots, Q_k$ contains all vertices in $U_2$, we have that $$\sum_{j = 1}^{k} (b(j) - a(j)) \ge |U_2| = 8.$$ We denote by $q_1$ and $q_2$ the numbers of 1-good and 2-good components among $Q_1, \dots, Q_k$.
	
	Let $c_j$ be the number of bags of $Q_j$. Note that every bag $H$ of $Q_j$ does contain some component of $K - s - t$ that does not contain the endvertices of $H$. We have that $c_j$ bags of $Q_j$ contain (at least) $c_j$ distinct components of $K - s - t$. Moreover, it is not hard to see that no distinct components from $Q_1, \dots, Q_k$ can have any bag in common. Therefore $K - s - t$ has at least $\sum_{j = 1}^k (1 + c_j)$ components. 
	
	We consider the following two cases according to Lemma~\ref{lem:3or4max3}.
	
	\noindent\textbf{Case I.} If $K - s - t$ has at most four components such that at most two of them are 1-good even cacti and the rest are good even cacti, then, by Lemmas~\ref{lem:good} and~\ref{lem:1good}, $K - s - t$ has at least $\sum_{j = 1}^{k} (1 + c_j) \ge \sum_{j = 1}^{k} (b(j) - a(j)) - q_1 \ge 8 - 2 = 6$ components, which contradicts that $K - s - t$ has at most four components.
	
	\noindent\textbf{Case II.} If $K - s - t$ has at most three components such that one of them is a 2-good even cactus and the rest are good even cacti, then, by Lemmas~\ref{lem:good} and~\ref{lem:2good}, $K - s - t$ has at least $\sum_{j = 1}^{k} (1 + c_j) \ge \sum_{j = 1}^{k} (b(j) - a(j)) - 2 q_2 \ge 8 - 2 = 6$ components, which contradicts that $K - s - t$ has at most three components.
	
	Therefore we may conclude that $G$ has no spanning good even cactus with maximum degree at most three, and this completes the proof of Theorem~\ref{thm:atmost3}.

\subsection{Proof of Theorem~\ref{thm:no}} \label{sec:no}

	In this section we shall show that there is a Hamilton cycle in the prism over $G(D_n)$, but there is no spanning subgraph of $G(D_n)$ that is a good even cactus. Note that $B$ will be the fragment $D_n$ throughout this section.
	
	We first prove that $G(D_n) \square K_2$ is hamiltonian. Here we assume $n$ is odd, the case $n$ is even can be dealt with analogously. 
	
	We need the following result due to Ellingham, Salehi Nowbandegani and Shan~\cite{Ellingham2020}. 
	
	\begin{proposition}[{\cite[Theorem~2.3]{Ellingham2020}}] \label{pro:HinP}
		Let $Q$ be a good even cactus. The prism over $Q$ has a Hamilton cycle that contains the edge at $(v, \alpha) (v, \beta)$ for any $v \in V(Q)$ with $b_Q(v) = 1$.
	\end{proposition}
	
	As shown in Figure~\ref{fig:A}, $A$ has a spanning good even cactus in which $u_1$ and $u_3$ have block degree one. Hence it follows from Proposition~\ref{pro:HinP} that there exists a Hamilton cycle $H_A$ in $A \square K_2$ containing the edges $(u_1, \alpha) (u_1, \beta)$ and $(u_3, \alpha) (u_3, \beta)$. In the prism over $B$, we define $L_B$, $S_B$ and $\tilde{S}_B$ to be the graphs depicted in Figures~\ref{fig:LS}(a),~(b) and~(c), respectively; and $R_B$ and $\tilde{R}_B$ be the reflections of $S_B$ and $\tilde{S}_B$, respectively. Each of the graphs $L_B$, $S_B$, $\tilde{S}_B$, $R_B$ and $\tilde{R}_B$ is a union of vertex-disjoint paths that spans $B \square K_2$.
	One can readily verify that \begin{align*}
	(H_A^1 - (u_3^1, \alpha) (u_3^1, \beta) ) \cup \left( \bigcup_{i = 2}^7 (H_A^i - (u_1^i, \alpha) (u_1^i, \beta) - (u_3^i, \alpha) (u_3^i, \beta) ) \right) \cup (H_A^8 - (u_1^8, \alpha) (u_1^8, \beta) ) \\
	\cup \left( \bigcup_{i = 2, 4, 6} L_B^i \right)
	\cup S_B^1 \cup \{ (s, \beta)(w_n^1, \beta), (t, \alpha)(x_{2n + 1}^1, \alpha) \} \cup \tilde{S}_B^3 \cup \{(s, \beta)(w_n^3, \beta), (t, \beta)(x_{2n}^3, \beta)\} \\
	\cup R_B^5 \cup \{ (s, \alpha)(w_n^5, \alpha), (t, \beta)(x_{2n + 1}^5, \beta) \} \cup \tilde{R}_B^7 \cup \{(s, \alpha)(w_n^7, \alpha), (t, \alpha)(x_{2n}^7, \alpha)\} \\
	\end{align*} is a Hamilton cycle of $G \square K_2$.
	
	\begin{figure} [h!t]
		\centering
		\subfigure[]{\includegraphics[scale = 0.8]{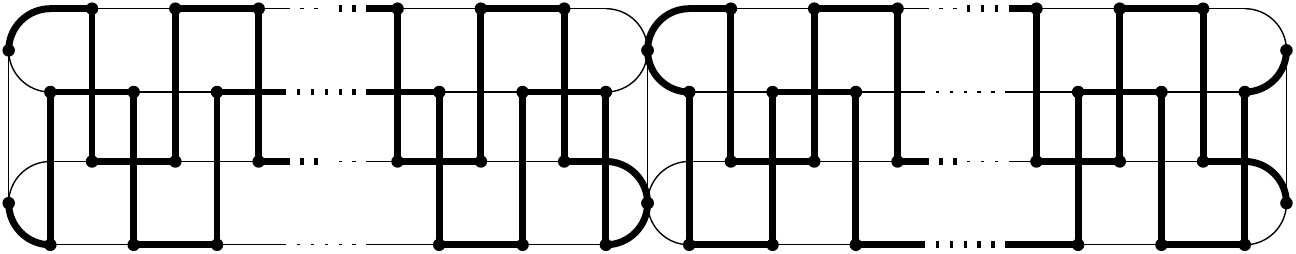}} \label{subfig:L}\\
		\subfigure[]{\includegraphics[scale = 0.8]{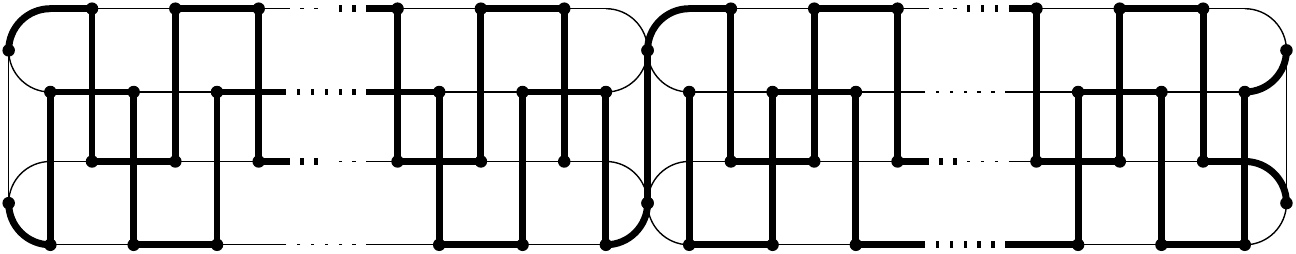}} \label{subfig:S}\\
		\subfigure[]{\includegraphics[scale = 0.8]{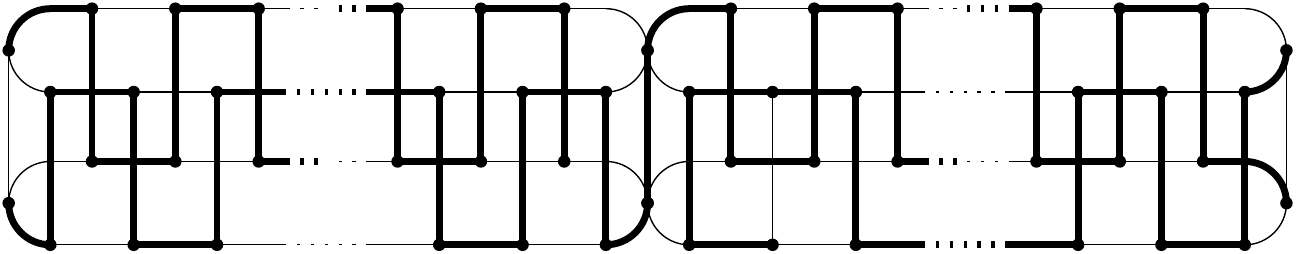}} \label{subfig:St}
		\caption{In each subfigure the thin and thick edges together represent the prism over $D_n$ such that the vertices of $V(D_n) \times \{\alpha\}$ are placed above that of $V(D_n) \times \{\beta\}$; the leftmost and rightmost straight edges denote $(l, \alpha) (l, \beta)$ and $(r, \alpha) (r, \beta)$, respectively; and each copy of $D_n$ is embedded in the same way as depicted in Figure~\ref{fig:CD}(b). (a) The graph $L_{D_n}$ (thick edges) consists of one path with endvertices $(l, \alpha)$ and $(l, \beta)$ and one with endvertices $(r, \alpha)$ and $(r, \beta)$. (b) The graph $S_{D_n}$ (thick edges) consists of one path with endvertices $(l, \alpha)$ and $(w_n, \beta)$, one with endvertices $(l, \beta)$ and $(r, \beta)$ and one with endvertices $(r, \alpha)$ and $(x_{2n + 1}, \alpha)$. (c) The graph $S_{D_n}$ (thick edges) consists of one path with endvertices $(l, \alpha)$ and $(w_n, \beta)$, one with endvertices $(l, \beta)$ and $(r, \beta)$ and one with endvertices $(r, \alpha)$ and $(x_{2n}, \beta)$.
			} 
		\label{fig:LS}
	\end{figure}
	
	It is now left to show that $G$ does not have any spanning good even cactus. Suppose that $G$ has a spanning good even cactus $K$ (reductio ad absurdum). As in the proof of Theorem~\ref{thm:atmost3}, we consider  the components $Q_1, \dots, Q_k$ of $K - s - t$ that contain some vertex from $U_2 := \{u_2^1, \dots, u_2^8\}$.
	For any component $Q_j$ of $K - s - t$ ($j \in \{1, \dots, k\}$), we choose $0 \le a(j) \le b(j) \le 8$ such that $Q_j$ is contained in $G^-[l^{a(j)}, r^{b(j)}]$ and, subject to this, $b(j) - a(j)$ is minimum. Again, the inequality $\sum_{j = 1}^{k} (b(j) - a(j)) \ge 8$ holds. Let $q_1$ and $q_2$ be the numbers of 1-good and 2-good cacti among $Q_1, \dots, Q_k$, respectively.
	
	As we have discussed in the previous section, $K - s - t$ has at least $\sum_{j = 1}^k (1 + c_j)$ components, where $c_j$ is the number of bags of $Q_j$. By Lemma~\ref{lem:3or4}, we have the following two cases.
	
	\noindent\textbf{Case I.} If $K - s - t$ consists of at most four even cacti which are good or 1-good, then, by Lemmas~\ref{lem:good} and~\ref{lem:1good}, $K - s - t$ has at least $\sum_{j = 1}^{k} (1 + c_j) \ge \sum_{j = 1}^{k} (b(j) - a(j)) - q_1 \ge 8 - 4 = 4$ components. Since there are at most four components in $K - s - t$, the equality must hold. In this case we must have that $q_1 = 4$. This implies that $k = 4$ and every component $Q_j$ is 1-good. Then we have $c_j = 0$ for every $j$. In other words, no $Q_j$ can have any bag. As $0 = c_j \ge b(j) - a(j) - 2$ holds for every $j$ and $\sum_{j = 1}^{4} (b(j) - a(j)) \ge 8$, we have that $b(j) = a(j) + 2$ for every $j \in \{1, \dots, 4\}$. We may assume that $a(j) = 2(j - 1)$ for any $j$, and that $Q_1$ and $Q_2$ are $P_1$-good and $P_2$-good, respectively. By Lemma~\ref{lem:1good}, $P_1$ and $P_2$ are contained in $G^-[r^0, l^2]$ and $G^-[r^2, l^4]$, respectively. If either $Q_1$ or $Q_2$ does not intersect $B^2$, then $(Q_1 \cup Q_2) [B^2]$ is a spanning good even cactus of $B^2$, indeed, an edge path with $l^2$ or $r^2$ has block degree one, which is clearly impossible. If both $Q_1$ and $Q_2$ intersect $B^2$, then $(Q_1 \cup Q_2) [B^2]$ is a union of two edge paths, one of which has $l^2$ as endvertex and the other has $r^2$ as endvertex, spanning the graph $B^2$, which is, again, impossible.
	
	\noindent\textbf{Case II.} If $K - s - t$ has at most three components such that one of them is a 2-good even cactus and the rest are good or 1-good even cacti, then, by Lemmas~\ref{lem:good},~\ref{lem:1good} and~\ref{lem:2good}, $K - s - t$ has at least $\sum_{j = 1}^{k} (1 + c_j) \ge \sum_{j = 1}^{k} (b(j) - a(j)) - q_1 - 2 q_2 \ge 4$ components, which contradicts that $K - s - t$ has at most three components.
	
	Hence we conclude that no spanning subgraph of $G$ can be a good even cactus, and this completes the proof of Theorem~\ref{thm:no}.
	
\section*{Acknowledgements}

The author wishes to thank Carol T.\ Zamfirescu for helpful discussion.

\bibliographystyle{abbrv}
\bibliography{paper}
\end{document}